\documentclass[10pt,twoside]{article}

\usepackage{amssymb,amsmath,amsthm}
\usepackage{bbm} 
\usepackage{array}
\usepackage{enumerate}
\usepackage{indentfirst}

\newtheorem{theorem}{\indent Theorem}[section]
\newtheorem{proposition}[theorem]{\indent Proposition}
\newtheorem{lemma}[theorem]{\indent Lemma}
\newtheorem{corollary}[theorem]{\indent Corollary}

\theoremstyle{definition}\newtheorem{definition}[theorem]{\indent Definition}
\theoremstyle{remark}\newtheorem{remark}[theorem]{\indent Remark}

\newcommand{\veps}{\varepsilon}
\newcommand{\vp}{\varphi}

\newcommand{\E}{\mathbb{E}}
\renewcommand{\Pr}{\mathbb{P}}
\newcommand{\ind}{\mathbbm{1}} 

\title{Random cherry graphs}

\author{{ Tam\'as F. M\'ori, and  S\'andor Rokob}\\[1ex]
	 (ELTE E\"otv\"os Lor\'and University, Budapest, Hungary)
}

\date{4 July 2018}

\begin{document}

\maketitle

\begin{abstract}
Due to the popularity of randomly evolving graph processes, there exists a randomized 
version of many recursively defined graph models. This is also the case with the 
cherry tree, which was introduced by Buksz\'ar and Pr\'ekopa to improve Bonferroni 
type upper bounds on the probability of the union of random events. Here we consider
a substantially extended random analogue of that model, embedding it into a general 
time dependent branching process.
\end{abstract}

\noindent {\small \textit{Keywords}. Cherry tree; Crump--Mode--Jagers process; 
exponential growth; extinction. } 

\section{Introduction}

In \cite{BukszarPrekopa}, Buksz\'ar and Pr\'ekopa introduced the following recursively 
defined graph model, called cherry tree. Initially, there is only an adjacent pair of 
vertices, the only cherry tree on exactly two vertices. From a cherry tree one can 
obtain another cherry tree by adding a new vertex and two new edges connecting this 
new vertex to two already existing vertices. This new, length $2$ path between the two
already existing vertices called cherry. Note that, in spite of their name,  cherry trees 
are not trees in the usual graph theoretic sense, as they generally contain cycles.

Their purpose was to improve the Hunter--Worsley second order upper bound on the 
probability of union of random events (see e.g.  \cite{Worsley}). The 
main idea behind their reasoning was to extend the spanning tree in the Hunter--Worsley 
inequality to a cherry tree. In fact, the extension they used was the so called $t$-cherry 
tree, a particular type of cherry trees where the cherries are always added to 
adjacent old vertices.

Apart from its use for constructing probability bounds, the graph model itself was not further
studied. A few years later a randomized generalization of this model was introduced in 
\cite{Mori}, where random evolving $m$-ary trees were introduced
and examined. In the case $m=2$ this model reduces to a random $t$-cherry tree.

A slightly related model is analysed with martingale methods in \cite{BackhauszMori1}, 
\cite{BackhauszMori2}, and \cite{FazekasNoszalyPerecsenyi}. In the particular case where 
the parameters of the model are appropriately set ($p=1$), it also defines a kind of random 
evolving cherry tree.

Here we consider a more general version of the random $t$-cherry tree. For the sake 
of convenience, this generalized model will also be called cherry tree or cherry graph. 
The main direction on the generalization is twofold: firstly, an edge is allowed to grow 
randomly many cherries at the same time; secondly, the possibility of edge deletion
is introduced. This not only breaks the monotonicity of growth -- making the analyis much 
harder -- but results in a more realistic model. Moreover, the way of the deletion arises 
in a really natural way.

The paper is organized as follows. First we gave a precise definition of the generalized 
version of the random cherry tree. Then we introduce the continuous time version of the 
model, completed with a well known stochastic process, namely, the Crump--Mode--Jagers 
process, which then constitutes the backbone of the analysis hereafter. Using this 
continuous version, we establish several properties of the model, such as  the probability of 
extinction, the asymptotic number of vertices or edges, the evolution of the degree of a 
fixed vertex, and so on.

\section{Model}

In this section, we introduce the basic notations and define our model of interest. Before 
setting up the model we need some definitions.

\subsection{Discrete time}

Assume we have a graph with only one edge connecting two vertices. Now, 
add a new vertex to this graph. If it is connected to both endpoints of the given 
edge, these new edges (and the new middle vertex) are called \emph{a cherry 
(of the existing edge)}. Alternatively, if the new vertex is joined to only 
one randomly chosen endpoint of the existing edge,  the new edge (and 
the new vertex) is called \emph{a semi-cherry} or \emph{cherry stem (of the 
existing edge)}. 

Now, we have everything to describe the main object of our further examinations. 

\begin{definition}
\emph{The random cherry tree} is a graph process evolving in 
discrete time steps in the following way.
\begin{enumerate}
	\item Initially, there is only one edge with two vertices.
	\item In a general step, the current graph changes in exactly one way 
                 of the following options.
	\begin{enumerate}
		\item A randomly chosen edge is deleted.
		\item A random number of cherries or semi-cherries are joined 
                         to a randomly chosen edge (reproduction event). 
	\end{enumerate}
\end{enumerate}
\end{definition}

To get a well defined model we have to specify what randomly means.
For this purpose we need some more notations. Let $\mathcal{E}_n$ and 
$\mathcal{V}_n$ denote the sets of edges and vertices of the graph after $n$ steps, resp. At the 
$n$-th step, let $D_n(e)$ denote the event that $e \in\mathcal{E}_n$ becomes deleted, and 
$C_n(e)$ that cherries and semi-cherries are joined to $e \in\mathcal{E}_n$ (reproduction).
Note that these notations are meaningful only in the case when $\mathcal{E}_n$ is not empty.
In a reproduction event, let $\kappa_n$ and $\veps_n$ denote the random number of 
new vertices and edges added to the graph, resp. The pairs $(\kappa_1,\veps_1)$, 
$(\kappa_2,\veps_2),\,\dots$ are iid copies of a generic pair of positive 
integer valued random variables $(\kappa,\veps)$, where
\[
\veps=\sum_{i=1}^{\kappa}\vp_i,
\]
with iid summands $\vp_1,\,\vp_2,\,\dots$, independent of $\kappa$, that represent 
the amounts of new edges connecting the new vertices to the existing graph. Thus, 
\[
\Pr(\vp_i = 2) = p, \quad \text{and} \quad \Pr(\vp_i = 1) = 1 - p.
\]
We will suppose that $\kappa$ has an everywhere finite probability generating function 
$g_{\kappa}(z) = \E (z^{\kappa})$. Then $g_{\veps}(z)=\E(z^{\veps})<\infty$
for every $z \in \mathbb{R}$ as well; more precisely, $g_{\veps}(z)=g_{\kappa}(pz^2+(1-p)z)$.

Finally, let $\xi_n(e)$ denote the number of cherries and semi-cherries attached to edge
$e\in\mathcal{E}_n$ \emph{before} the $n$-th step (regardless that they still are in the graph or got
deleted at an earlier stage). This will be called \emph{the biological age} of edge 
$e$ at the $n$-th step.

Using these notations we are able to define the probabilities of the events $D_n(e)$ 
and $C_n(e)$:
\[
	\Pr(D_n(e)) = q_n (b + c \,\xi_n(e)), \quad \text{and} \quad \Pr(C_n(e)) = q_n,
\]
where $b,c$ are positive constants, and $q_n$ is a normalizing multiplier in order 
that the sum of probabilities of all these events be equal to $1$.

An alternative formulation can be given by introducing weights of edges. When an edge is 
added to the graph, initially it has weight $1+b$. Every cherry and semi-cherry connected to
an edge increases the edge's weight by $2c$ and $c$, resp. At each step, we first select 
an existing edge with probability proportional to its weight $w$, then either we delete it
with probability $1-1/w$, or reproduction takes place, with probability $1/w$.

So far we have given the mathematically rigorous definition of our random cherry graph, 
although it does not seem easy to treat. An obvious, and, as we will see, useful, idea is 
to change the time from discrete to continuous, examine the new continuous version, and 
then draw the appropriate conclusions on the original, embedded model. So, in the next 
subsection we will define this continuous version of our randomly evolving graph.

\subsection{Continuous time}

First of all ignore the fact that the time elapsed between consecutive events is considered 
as unit, and take a look at the role of the edges. After a new edge is drawn between a 
new and a previously existing vertex, it grows a random number of cherries and 
semi-cherries at the same time, on a random number of occasions before its deletion, 
which happens with probability proportional to a linear function of the number of these 
new edges. Hence, whenever new cherries and semi-cherries are joined to an edge, these 
new edges can be interpreted as descendants of the selected edge.

Accepting this approach, it is much easier to introduce the corresponding continuous time 
version of the previuosly defined random cherry graph. Furthermore, it can give the reader 
the idea, how our analysis will be performed in the forthcoming sections.

\begin{definition}
\emph{The continuous random cherry tree} is a graph process which is evolving in 
continuous time, as described below.
\begin{enumerate}
	\item Initially, there is only one edge, joining two vertices, called the ancestor.
	\item An edge produces possibly more than one new edges, as its children, at 
                different birth events, which form a homogeneous Poisson process of unit 
                density. Formally:
	\begin{enumerate}
		\item At every birth event a random number $\kappa$ of new vertices are 
                        added to the current graph. Their numbers are iid random variables.
		\item Each of these new vertices is connected to either a randomly chosen
                        endpoint of the selected edge with probability $1-p$, or both of its endpoints
                        with probability $p$.
	\end{enumerate} 
	\item To consider the deletion (or death) of an edge, let us call the time elapsed from its 
                birth \emph{the edge's physical age}, and let the number of new edges 
                 born up to physical age $t$ be denoted by $\xi(t)$ (this is the edge's biological age).
               	The edge is deleted at physical age $t$ with hazard rate $b + c \xi(t)$, a linear 
                function of its biological age. This means that the (conditional to the reproducing process)
                probability of surviving physical age $t$ is equal to
\[
\exp\bigg(-\int_0^t (b+c\xi(s))ds\bigg).
\]
\end{enumerate}
	Life histories of the different edges are assumed to be independent.
\end{definition}

Looking at this continuous time model, one can ask, how can this take us closer to the 
analysis of our original cherry graph. The answer is somehow hidden in the phrasing 
of its definition. Indeed, we used the words ancestor, children, birth and death to suggest 
that, in spite of its derivation, the described graph, looking at the process from the 
viewpoint of the edges, is nothing else than a \emph{Markov branching process}. 
Though the Markov property is a strong and profitable feature that a stochastic process 
can have, here we will only use the fact that this is just a special case of the so-called
 \emph{general time dependent branching process}, or \emph{Crump--Mode--Jagers 
(CMJ) process}.

\subsection{General branching processes}

Since there are several monographs discussing the properties of general branching processes 
(see e.g. \cite{Cohn}, \cite{CrumpModeI}, \cite{CrumpModeII}, \cite{Nerman}, or 
\cite{Haccouetal}, \cite{Jagers}),  
 here we only summarize how our model fits the theory of CMJ processes, 
borrowing the notations from \cite{MoriRokob1} and \cite{Nerman}.

Consider an arbitrary edge in our graph. Denote the Poisson process of its birth events 
by $(\pi(t))_{t \geq 0}$, and its life span by $\lambda$, with survival function $S(t) = 
\Pr(\lambda > t)$. At the consecutive birth events $\tau_i$ ($i = 1, 2, \dots$) this edge 
gives birth to $\veps_i$ ($i = 1, 2, \dots$) random edges, which are connected to its 
endpoints forming cherries or cherry stems. Hence the number of descendants of 
this parent edge up to the $i$-th birth event is equal to the sum $B_i = \veps_1 + 
\veps_2 + \dots + \veps_i$, thus its biological age at physical age $t$ is given by the 
random sum $\xi(t) = B_{\pi(t \wedge \lambda)}$. This defines a compound Poisson 
process. In the theory of general branching processes, the process $(\xi(t))_{t \geq 0}$ 
is called \emph{the reproduction process}.

Although all individuals $e$ in the general branching process can be characterized by the 
pairs $(\lambda_e, \xi_e)$, which are iid copies of $(\lambda, \xi(.))$ defined above, the 
popularity of this model is due to a third process joined to these two. This stochastic 
process, often denoted by $\phi(.)$, is called \emph{a random characteristic}. It somehow 
takes the history of an individual into consideration. In most cases it is assumed that 
$\phi(t) = 0$ if $t \leq 0$, and $\phi(t) \equiv \phi(\lambda)$ whenever $t \geq \lambda$, 
but it is not necessarily required.

Complete the previously defined pairs with iid copies of the random characteristic and 
denote the birth time of edge $e$ by $\sigma_e$. Then, summing up $\phi_e(t - \sigma_e)$ 
over all edges, namely, taking the sum
\[
Z^{\phi}(t) = \sum_{e} \phi_e(t - \sigma_e),
\]
only those individuals are counted, who are alive and possess the property measured by 
$\phi(\cdot)$ at the given moment. Accordingly, the process $(Z^{\phi}(t))_{t \geq 0}$ is 
called \emph{a (time-dependent) branching process counted by a random characteristic}. 
To enlight this notion, consider the random characteristic $\phi(t) = \ind_{\{ 0 \leq t < 
\lambda \}}$ as an example. Then the branching process $(Z^{\phi}(t))_{t \geq 0}$ counted 
by this characteristic is nothing else than the number of living individuals at time $t$.

Using the notations introduced above, it is obvious that edge $e$ is deleted at 
time $\sigma_e + \lambda_e$. It is possible that our graph process dies out, 
i.e., eventually all the edges get erased. Furthermore, it is well-known (see \cite{Jagers}),
that the reproduction mean $\E\xi(\infty)$ plays crucial role in the characterization of extinction 
(similarly to the discrete time Galton--Watson processes). Indeed, if this mean is less than
 or equal to $1$ (subcritical and critical regimes), then the process dies out almost surely. 
On the other hand, when $\E \xi( \infty) > 1$ (supercritical case), the extinction probability 
is strictly less than $1$. From now on, we only deal with the latter case, restricting ourselves 
to the event where the process does not get extinct.

In the continuous model the underlying branching process grows exponentially fast on 
the event of non-extinction, and the growth rate is characterized by the \emph{Malthusian 
parameter}, denoted by $\alpha$. This is the only positive solution of the equation
\begin{equation}\label{Malthus}
	\int_0^{\infty} e^{- \alpha t} \, \mu(dt) = 1,
\end{equation}
where $\mu$ is the so called \emph{intensity measure}, defined by $\mu(t) = \E \xi(t)$. 
With these concepts and notations, we have everything needed to cite the theorem proved 
by Nerman in \cite{Nerman}, which shows the asymptotic properties of a general branching 
process counted by a random characteristic. Since we do not need the most general form, 
here we only cite the form stated in \cite{MoriRokob1}.

\begin{theorem}\label{Nerman}
	Suppose the random characteristic $\phi$ satisfies the following conditions:  
	\begin{enumerate}[(i)]
		\item $\phi(t)\ge 0$,
		\item the trajectories of $\phi$ belong to the Skorokhod $D$-space, that is, 
                        they do not have discontinuities of the second kind,
		\item $\E[\sup_t\phi(t)]<\infty$.
	\end{enumerate}
	Furthermore, with the definition
	\[
		{}_{\alpha} \xi(t) = \int_0^t e^{- \alpha s} \, \xi(ds),
	 \]
	we have ${}_{\alpha} \xi(\infty) \in L \log^+ L$. Then 
	\begin{equation}\label{limphi}
		\lim_{t\to\infty} e^{- \alpha t} Z^{\phi}(t)=Y_\infty m^{\phi}_{\infty}
	\end{equation}
	almost surely, where 
	\begin{equation}\label{mphiinfty}
	m^{\phi}_{\infty}=\frac{\int_0^\infty e^{-\alpha t}\E\phi(t)\,dt}{\int_0^\infty t\, 
        e^{-\alpha t}\mu(dt)}\,,
	\end{equation} 
	$Y_\infty$ is a nonnegative random variable, which is positive
	on the event of non-extinction, it has expectation $1$, and it does not
	depend on the choice of $\phi$.
	
	In addition, if the random characteristics $\phi$ and $\psi$  are both satisfying 
        the conditions above, then, almost everywhere on the event of non-extinction, 
	\begin{equation}\label{fracphipsi}
	\lim_{t \to \infty} \frac{Z^{\phi}(t)}{Z^{\psi}(t)} = 
        \frac{ \int_0^{\infty} e^{- \alpha t} \E \phi(t) \, dt}{ \int_0^{\infty} e^{- \alpha t} 
        \E \psi(t) \, dt }
	\end{equation}
	holds.
\end{theorem}

Using the statements of this theorem, we will be able to rigorously formulate the 
connection between the discrete and the continuous models. Then, it will be relatively 
easy -- again with the help of Theorem \ref{Nerman} -- to describe certain properties 
of the discrete time graph, by proving results for the continuous one.

\section{Properties}

The section is organized as follows. First of all, we show how the growth rate in the 
discrete modell is connected to that of the continuous one. After that, taking only the 
latter model into consideration, we can deduce results on the original random cherry tree. 
This upcoming collection of propositions will include the probability of extinction, the 
asymptotic number of vertices, and other properties.

Before examinig the two models' connection, let us make some remarks. In Theorem 
\ref{Nerman} we have introduced the intensity measure $\mu$. By definition, $\mu(t)$ 
is the mean reproduction at time $t$. By applying Wald's identity we can 
express it in terms of the lifespan's survival function:
\[
\mu(t) = \E \xi(t) = \E(\veps) \E (\lambda \wedge t) = \E (\veps) \int_0^t S(u) \, du.
\]
Thus, the equation for the Malthusian parameter takes the following shape:
\begin{equation}\label{intensity}
 \E (\veps)\int_0^\infty e^{-\alpha t}S(t)\,dt=1.
\end{equation}

The idea of the present random cherry tree model comes from the continuous time random graph 
model considered in \cite{MoriRokob2}. Though there we ``\textit{did not fix how many 
new edges should be added to the graph, or how the subgraph they form should look like}", 
some important properties could be proved without further specification. Here we cite them 
merged into one theorem.
\begin{theorem} \cite[Corollaries $3.1$, $3.2$, and $3.3$]{MoriRokob2}
	\vspace{-5ex}
	\begin{description}
		\item[]
		\item[Survival function.] The survival function of the lifespan statisfies
		\begin{equation}\label{S}
		S(t) = P( \lambda > t ) = \exp \bigg\{ -(1+b) t + \frac{1}{c} \int_{e^{-ct}}^1 
                \frac{g_{\veps}(v)}{v} \, dv \bigg\}.
		\end{equation}
		\item[Supercriticality.] The random cherry tree is supercritical $(\E \xi(\infty) > 1)$ 
                if and only if
		\begin{equation}\label{supercritical}
		\frac{ \E (\veps)}{c} \int_0^1 u^{\tfrac{1+b}{c}-1} \exp \bigg\{ \frac{1}{c} 
                \int_u^1 \frac{g_{\veps}(v)}{v} \, dv \bigg\} \, du > 1.
		\end{equation}
		\item[Malthusian parameter.] The Malthusian parameter $\alpha$ of the 
                continuous time random cherry tree is determined by the equation
		\begin{equation}\label{Malthusp}
			\frac{ \E (\veps)}{c} \int_0^1 u^{\tfrac{\alpha+1+b}{c}-1} \exp \bigg 	\{
			\frac{1}{c} \int_u^1 \frac{g_{\veps}(v)}{v} \, dv \bigg\} \, du = 1.
		\end{equation}
	\end{description}
\end{theorem}

\subsection{From discrete to continuous}

As mentioned before, the original random cherry tree model is embedded into the continuous 
one. Indeed, if one takes `snapshots' of the continuous random cherry tree at the moments of 
events (which can be birth or death) and looks at these photographs one by one in chronological 
order, then the resulted process is just the discrete time cherry tree process. 

As a consequence, it is obvious that the probability of extinction is the same for both processes. 
However, so as to transfer the asymptotic results that will be obtained for the continuous case,
we need to compare the growth rates of the two processes. In order to do so, as a first step, we will 
calculate the asymptotics of the number of edges in the continuous model. 

Recall the definition 
\[
{}_{\alpha} \xi(t) = \int_0^t e^{- \alpha s} \, \xi(ds),
\]
where $\alpha$ is the Malthusian-parameter, and $(\xi(t))_{t \geq 0}$ is the biological age process 
of an edge. 

\begin{proposition}\label{xi_cond}
\[
\E \big[ {}_{\alpha} \xi^2(\infty) \big] < \infty,
\]
and hence the condition ${}_{\alpha}\xi(\infty) \in L \log^+ L$ is statisfied.		
\end{proposition}

\begin{proof}	
By definition we have
\[
{}_{\alpha} \xi(\infty) = \int_0^{\infty} e^{- \alpha t} \, \xi (dt) = 
\sum_{\tau_i < \lambda} \veps_i e^{- \alpha \tau_i} 
\leq \sum_{i = 1}^{\infty}  \veps _i e^{- \alpha \tau_i}.
\]
Note that the random variables $\veps_i$ and $\tau_i$ are independent for every $i=1,2,\dots$. 
Hence, for the $L^2$ norm we get
\[
{\|{}_{\alpha} \xi(\infty)'\|}_2 \le \sum_{i = 1}^{\infty} {\|\veps_i e^{- \alpha \tau_i}\|}_2 
= {\|\veps\|}_2 \sum_{i=1}^{\infty} \frac{1}{(1+2\alpha)^{i/2}} <\infty,
\]
using the fact that in a homogeneous Poisson process with unit density the birth times 
$\tau_i$ are distributed as $\Gamma(i, 1)$. 
\end{proof}

\begin{proposition}\label{living_edges}
Denote the number of living edges in the continuous model at time $t$ by $E(t)$. Then 
\[
\lim_{t \to \infty} e^{- \alpha t} E(t) = \Big[ \E^2(\veps) \int_0^\infty t\,
e^{-\alpha t} S(t) \, dt \Big]^{-1} Y_{\infty}
\]
almost surely, where $Y_{\infty}$ is the same as in Theorem \ref{Nerman}.
\end{proposition}

\begin{proof}
Since $E(t) = Z^{\phi}(t)$ with the random characteristic $\phi(t) =\ind_{\{ 0 \le t<\lambda\}}$, 
it is plausible to use Theorem \ref{Nerman}. Hence the proof of the statement is conducted by 
showing that all assumptions imposed on $(\phi(t))_{t \geq 0}$ hold, and then calculating the 
constant $m_{\infty}^{\phi}$. However, since this random characteristic is just an indicator, 
the conditions are trivially satisfied, so it is enough to determine the constant, which, by the 
definitions and the previous remark on the intensity measure, is equal to
\begin{gather*}
m_{\infty}^{\phi} = \frac{\int_0^\infty e^{-\alpha t}\E\phi(t)\,dt}
{\int_0^\infty t\,e^{-\alpha t}\mu(dt)} = \frac{\int_0^\infty e^{-\alpha t} S(t) \,dt}
{ \E (\veps) \int_0^\infty t\,e^{-\alpha t} S(t) \, dt } \\
= \bigg[ \E^2(\veps) \int_0^\infty t\,e^{-\alpha t} S(t) \, dt \bigg]^{-1}.
\qedhere
\end{gather*}
\end{proof}

In order to transfer this result to the original discrete time cherry tree, we have to deal
with the asymptotic growth rate of the number of events in the continuous time model.

\begin{theorem}\label{growthrate} 
Introduce the notation $H(t)$ for the number of events (birth or death) up to time $t$. 
Then, on the event of non-extinction,
\[
\lim_{t \to \infty} \frac{H(t)}{E(t)} = \frac{ \E (\veps) + 1-\alpha}{ \alpha}
\]
almost everywhere.
\end{theorem}

\begin{proof}
We want to use \eqref{fracphipsi} from Theorem \ref{Nerman}. Since the asymptotics of
$E(t)$ is known from Proposition \ref{living_edges}, it is enough to find an adequate random 
characteristic $\psi(t)$ for which $H(t) = Z^{\psi}(t)$ holds. It is obvious that $\psi(t) =
\pi(t \wedge \lambda)+\ind_{\{ \lambda \leq t \}}$ will do (note that $\pi(t)=0$ for 
negative $t$).
	
To compute the numerator of \eqref{fracphipsi}, recall that $(\pi(t))_{t \geq 0}$ is a 
Poisson process with unit intensity, hence 
\[
\E (\pi(t \wedge \lambda)) = \E(t \wedge \lambda) = \int_0^t S(u) \, du.
\]
Reversing the order of integrations we get
\begin{equation}\label{kellenifog}
\int_0^\infty e^{-\alpha t}\E(t \wedge \lambda) =\int_0^\infty\int_u^\infty
 e^{-\alpha t}\,dt \,S(u)\,du=\frac{1}{\alpha}\int_0^\infty e^{-\alpha u}S(u)\,du.
\end{equation}
This, by \eqref{intensity} and \eqref{Malthus}, implies that
\begin{align*}
\int_0^{\infty} e^{- \alpha t} \E(\psi(t) )\, dt & = 
\frac{1}{\alpha} \int_0^{\infty} e^{- \alpha t} S(t) \, dt + 
\int_0^{\infty} e^{- \alpha t} (1 - S(t)) \, dt \\ 
&= \int_0^\infty e^{-\alpha t}\,dt + \Big( \frac{1}{\alpha} - 1 \Big) \int_0^{\infty} 
e^{- \alpha t} S(t) \, dt \\ 
& =	\frac{1}{\alpha} + \Big( \frac{1}{\alpha} - 1 \Big) \frac{1}{ \E (\veps)} 
\int_0^{\infty} e^{- \alpha t} \, \mu(dt) \\ 
& = 	\frac{1}{\alpha} + \Big( \frac{1}{\alpha} - 1 \Big) \frac{1}{ \E (\veps)} 
\end{align*}
holds. Plugging this, and the result of Proposition \ref{living_edges} into 
\eqref{fracphipsi}, we get
\begin{gather*}
 \frac{ \int_0^{\infty} e^{- \alpha t} \E \big[\ind_{\{ \lambda \leq t \}}+ 
 \pi(t \wedge \lambda) \big] \, dt }{ \int_0^{\infty} e^{- \alpha t} \E \big[
 \ind_{\{ 0 \leq t < \lambda \}} \big] \, dt } = \frac{ \tfrac{1}{\alpha} + 
 \big( \tfrac{1}{\alpha} - 1 \big) \tfrac{1}{ \E (\veps)}}
 {\int_0^\infty e^{-\alpha t} S(t) \,dt}\\
 = \frac{ \E (\veps) + 1-\alpha}{ \alpha}\,,
\end{gather*}
as needed.
\end{proof}

It is evident that we can obtain results on the original discrete time cherry tree, 
if we normalize a quantity of the continuous one with the number of events $H(t)$. 
For example, Theorem \ref{growthrate} immediately yields the following asymptotic 
property of the number $E_n$ of  living edges in the discrete time cherry tree.

\begin{corollary}\label{edges}
Almost everywhere on the event of non-extinction,
\[
\lim_{n \to \infty} \frac{E_n}{n} = \lim_{t \to \infty} \frac{E(t)}{H(t)}
= \frac{\alpha}{  \E (\veps) + 1 - \alpha }\,,
\]
where $E_n=|\mathcal{E}_n|$  denotes the number of edges after the $n$-th step.\qed
\end{corollary}

\subsection{Probability of extinction}\label{ss3.2}

Inequality \eqref{supercritical} contains a necessary and sufficient condition 
for our evolving graph process to be
supercritical. In this case the probability of extinction (when all edges die out) 
is strictly less than $1$. Since there is an embedded Galton-Watson process 
with offspring size $\xi(\infty)=\xi(\lambda)$ in every general branching process, 
this probability can be obtained as the smallest nonnegative solution of the 
equation $g_{\xi(\lambda)}(z) = z$, where 
\[
g_{\xi(\lambda)}(z) = \E \big( z^{\xi(\lambda)} \big).
\]
Hence, for the extinction probability we need to compute this probability generating 
function. In the next lemma we derive a general formula, from which the 
requested generating function can easily be obtained. To that end, we first 
introduce the process $(\pi'(t))_{t \geq 0}$ that counts the number of 
vertices added to the graph by an edge up to and including its physical age $t$. 
This is a compound Poisson process having jumps exactly when so does $\pi(t)$. 
The jump sizes are $\kappa_i$ (iid copies of $\kappa$). 
\begin{lemma}\label{genfunc}
Define the joint probability generating function of $\pi'(\lambda)$ and $\xi(\lambda)$ as
\[
f(x,y) = \E \big( x^{\pi'(\lambda)} y^{\xi(\lambda)} \big) = \sum_{i=0}^{\infty} 
\sum_{j=i}^{2i} \Pr(\pi'(\lambda) = i, \, \xi(\lambda) = j) x^i y^j.
\]
Then
\[
f(x,y) = 1 - \frac{1 -  g_{\kappa, \veps}(x,y)}{m} \int_0^1 u^{ \tfrac{1+a}{m} - 1 } 
\exp \bigg\{ \int_u^1 g_{\kappa, \veps} (x, sy) \, ds \bigg\} du,
\]
where $g_{\kappa, \veps}(x,y)= \E \big( x^{\kappa} y^{\veps}  \big)$ is the joint 
probability generating function of $(\kappa, \veps)$.
\end{lemma}
\begin{remark}
Using the well-known formula for the generating function of the binomial distribution  
and the connection between $\kappa$ and $\veps$, we obtain
\[
g_{\kappa, \veps} (x,y) = \E\big(x^{\kappa}\,\E(y^\veps|\kappa)\big)
 =\E(x^\kappa (py^2+(1-p)y)^\kappa)= g_{\kappa} \big( x y (py + (1-p))\big)
\]
for the joint probability generating function of $(\kappa, \veps)$.
\end{remark}

\begin{proof}[Proof of Lemma \ref{genfunc}]
First, consider the generating function
\[
G(x,y) = \sum_{i=0}^{\infty} \sum_{j=0}^{i} \frac{ \Pr( \exists \, t < \lambda : 
\pi'(t) = i, \, \xi(t) =i+ j ) }{ 1 + b + c(i+j) } x^i y^j.
\]
Since $\pi'(t)\le\xi(t)\le 2\pi'(t)$, it seems reasonable to deal with events of the form 
$\{ \exists \, t < \lambda : \pi'(t) = i, \, \xi(t) =i+ j\}$.
For the sake of convenience, denote the coefficient of $x^iy^j$  by $v_{i, i+j}$. 
Note that $v_{i,i+j}=0$ for $j>i$.  By the definition of our process $v_{0,0} = 
\tfrac{1}{b+1}$ and $v_{0,j} = 0 $ ($j \geq 1$). Since
\begin{align*}
\Pr( \exists t < \lambda :\ &\pi'(t) = i, \, \xi(t) = i+j )\\
&= \sum_{\ell = 1}^{i} \sum_{k = 0}^{j}\Pr( \exists t < \lambda : \pi'(t) = i - \ell,
 \, \xi(t) = i-\ell+j -k)\\ &\quad\times \frac{\Pr(\kappa = \ell,\,\veps =\ell+k)}
{1 + b + c(i-\ell+j-k)}\,,
\end{align*}
with the notation introduced above we have the following recursion:
\[
(1+b+c(i+j)) v_{i,i+j}=\sum_{\ell =1}^{i} \sum_{k = 0}^{j} v_{i-\ell,(i-\ell)+(j-k)} 
\Pr(\kappa = \ell, \veps = \ell+k).
\]
Multiply both sides by $x^i y^j$ and add up for $i\geq 1$, $0\le j\le i$ to get
\begin{align*}
& (1+b)\Big[G(x,y)-\frac{1}{1+b}\Big]+c(x\,\partial_xG(x,y)+y\,\partial_yG(x,y))\\ 
& \qquad =\sum_{i=1}^{\infty}\sum_{j=0}^{i}\sum_{\ell=1}^{i}\sum_{k=0}^{j}  
v_{i-\ell,(i-\ell)+(j-k)}\Pr(\kappa = \ell, \veps = \ell+k) x^i y^j \\ 
& \qquad =\sum_{\ell=1}^{\infty}\sum_{k=0}^{\ell}\Pr(\kappa=\ell,\veps=\ell+k) 
x^{\ell} y^{k} \sum_{i = \ell}^{\infty} \sum_{j = k}^{i} v_{i-\ell,(i-\ell)+(j-k)} 
x^{i-\ell} y^{j-k} \\ 
&\qquad = g_{\kappa,\veps}\big(\tfrac xy,\,y\big)G(x,y).
\end{align*}
After rearrangement, we obtain the following partial differential equation:
\[
\left\{
\begin{array}{ll}
	\big[1+ b - g_{\kappa,\veps}\big(\tfrac xy,\,y\big)\big] G(x,y) + 
	c(x\, \partial_x G(x,y) + y\,\partial_y G(x,y) ) = 1;\\[6pt]
	G(0,y) = \frac{1}{1+b}\,.
\end{array}
\right.
\]
	
To solve this equation, we introduce the function $h(t) = G(tx, ty)$, which then satisfies 
the following linear ODE.
\begin{align*}
\left\{
\begin{array}{ll}
	\big[1+ b - g_{\kappa,\veps}\big(\tfrac xy,\,ty\big)\big] h(t) + c t h'(t) = 1;\\[6pt]
	h(0) = \frac{1}{1+b}\,.
\end{array}
\right.
\end{align*}
Such an ODE is a routine problem to solve, and its solution is
\[
h(t) = \frac{1}{c}\, t^{- \tfrac{1+b}{c}} \int_0^t u^{ \tfrac{1+b}{c} - 1 } 
\exp \bigg\{ \int_u^t g_{\kappa,\veps}\big(\tfrac xy,\,sy\big) \, ds \bigg\} du.
\]
The correspondence between $h(t)$ and $G(x,y)$ yields
\[
G(x,y) = h(1)= \frac{1}{c} \int_0^1 u^{ \tfrac{1+b}{c} - 1 } \exp \bigg\{ 
\int_u^1 g_{\kappa,\veps}\big(\tfrac xy,\,sy\big) \, ds \bigg\} du.
\]
	
Let us turn to the bivariate generating function $f$ we are interested in. Clearly,
\[
f\Big(\frac xy,\,y\Big)=\sum_{i=0}^{\infty} \sum_{j=0}^{i} \Pr(\pi'(\lambda)=i, 
\, \xi(\lambda) = i + j ) x^i y^j.
\]
The probability that an edge of biological age $j$ dies before the next reproduction 
event is 
\[
\frac{b + cj}{1 + b + cj}\,,
\]
therefore
\begin{align*}
\Pr(\pi'(\lambda)=i,\, &\xi(\lambda) = i + j )\\
&=\Pr( \exists \, t < \lambda : \pi'(t) = i, \, \xi(t) =i+ j )\,\frac{b + c(i+j)}
{1 + b + c(i+j)}\\
&=[b + c(i+j)]v_{i,i+j}.
\end{align*}
Consequently,
\begin{align*}
f\Big(\frac xy,\,y\Big)&=\sum_{i=0}^{\infty} \sum_{j=0}^{i}[b+c(i+j)]
v_{i, i+j} x^i y^j \\
&= b\, G(x,y) + c \big( x\,\partial_x G(x,y) + y\,\partial_y G(x,y) \big).
\end{align*}
As we have already seen,
\[
c \big(x\,\partial_x G(x,y) +y\,\partial_y G(x,y) \big)=1- \big[1+ b - g_{\kappa,\veps}
\big(\tfrac xy,\,y\big)\big] G(x,y),
\]
which implies
\[
f\Big(\frac xy,\,y\Big)= 1 -\big[1 - g_{\kappa,\veps}\big(\tfrac xy,\,y\big)\big] G(x,y).
\]
Finally, so as to get the generating function $f(x,y)$ we simply have to replace $x$ 
with $xy$. Then we conclude with
\[
f(x,y) = 1 - [1 -  g_{\kappa, \veps}(x,y)] G(xy,y).\qedhere
\]
\end{proof}

Substituting $1$ for $x$ we get the generating function of $\xi(\lambda)$. Note 
that $g_{\kappa,\veps}(1,y)=g_{\veps}(y)$.

\begin{corollary}
The probability generating function of $\xi(\lambda)$ is
\[
g_{\xi(\lambda)}(z)= 1 - \frac{1 -  g_{\veps}(z)}{c} 
\int_0^1 u^{ \tfrac{1+b}{c} - 1 } \exp\bigg\{ \int_u^1 g_{\veps}(sy)\,ds 
\bigg\} \, du.\tag*{\qed}
\]
\end{corollary}

Turning back to the theory of general branching proesses, we can determine
the probability that our random cherry tree eventually dies out, when the reproduction 
mean  $\E(\xi(\lambda))$ is greater than $1$.

\begin{corollary}
Assume that \eqref{supercritical} holds. Then the probability that the random 
cherry tree process becomes extinct is equal to the smallest nonnegative root 
of the equation
\[
\frac{1 -  g_{\veps}(z)}{c} \int_0^1 u^{ \tfrac{1+b}{c} - 1 } \exp\bigg\{ 
\int_u^1 g_{\veps}(sy)\,ds \bigg\} du = 1 - z.
\]
\end{corollary}
\begin{proof}
The probability of extinction is the smallest nonnegative root of the fixpoint equation
$g_{\xi(\lambda)}(z) = z$.
\end{proof}

\subsection{Asymptotics of vertices}\label{ss3.3}

This section focuses on the vertices in the cherry tree. First we consider the number 
of vertices, then turn our attention to how the degree of a fixed vertex changes in
time. We deal with the former question in a similar way as in the proof of 
Proposition \ref{living_edges}, meanwhile the latter one needs a slightly more work.

\begin{proposition}
Let $V_n=|\mathcal{V}_n|$. Almost everywhere on the event of non-extinction,
\[
\lim_{n \to \infty} \frac{V_n}{n} =\frac{\E(\kappa)}{\E(\veps)+1-\alpha}\,.
\]
\end{proposition}
\begin{proof}
It is obvious that
\[
\lim_{n \to \infty} \frac{V_n}{n} = \lim_{t \to \infty} \frac{V(t)}{H(t)}
\]
holds. Furthermore, from the proof of Theorem \ref{growthrate} we know that 
$H(t) = Z^{\phi}(t)$, where
\[
\phi(t) = \pi(t \wedge \lambda) + \ind_{\{ 0 \leq t < \lambda \}},
\]
for which
\[
\int_0^{\infty} e^{- \alpha t} \, \E\phi(t) \, dt = \frac{1}{\alpha} + 
\Big( \frac{1}{\alpha} - 1 \Big) \frac{1}{ \E (\veps)}=\frac{\E(\veps)+1-\alpha}
{\alpha\E(\veps)}\,.
\]
Hence, we only need to find a characteristic $\psi$ that counts the vertices in the 
graph and then compute the corresponding integral $\int_0^{\infty}e^{-\alpha t}
\,\E\psi(t)\,dt$. Set $\psi(t) = \pi'(\lambda \wedge t)$ (the compound Poisson
process $\pi'(\cdot)$ was introduced in Subsection \ref{ss3.2}), then $Z^{\psi}(t)$ 
is less than $V(t)$ by the two initial nodes, whose significance asymptotically 
vanishes. Now we get
\[
\E\psi(t) = \E(\kappa) \E\pi(t \wedge \lambda) = \E(\kappa) \int_0^t S(u) \, du,
\]
and consequently, 
\[
\int_0^{\infty} e^{- \alpha t} \E \psi(t) \, dt = \frac{\E(\kappa)}{\alpha} 
\int_0^{\infty} e^{- \alpha t} S(t) \, dt = \frac{1}{\alpha (1+p)},
\]
using \eqref{Malthus}. Using that  $\E(\veps)=(1+p)\E(\kappa)$, by Theorem \ref{Nerman} 
we obtain
\[
\lim_{t \to \infty}\frac{V(t)}{H(t)}=\frac{\E(\kappa)}{\E(\veps)+1-\alpha}\,,
\]
as requested.
\end{proof}

It is obvious (or, using the random characteristic $\xi(t) - \pi'(t)$, easy to check) that the 
asymptotic proportion of vertices born with two edges is equal to $p$. In this way we
can get the asymptotic number of triangles ever created in the graph. However, if we 
wanted to examine the number of \emph{living} triangles only, the resulting formula would be 
much more complicated.

Next, let us turn to the behaviour of the degree process of a fixed vertex. To handle this 
problem, we need to define a new branching process, closely related to the original one 
constructed on the edges, which only takes the edges joined to the given vertex into 
consideration. To this end, suppose that the fixed vertex is born with a single initial edge 
and introduce the following notations. 

Whenever an edge, joined to the vertex 
under consideration, gives birth to $\kappa$ cherries and semi-cherries, each can increase 
the degree of the monitored vertex by $1$; namely, a cherry will always increase, but a 
semi-cherry only with probability $1/2$. Let $\phi_i$ be equal to $2$ if the contribution of the 
$i$-th new vertex is a cherry, and $1$, if  it is a semi-cherry. Let $\gamma_1, \gamma_2, 
\dots$ be iid with conditional distribution
\[
P(\gamma_i = 1 | \varphi_i = 2) = 1; \quad P(\gamma_i = 1 | \varphi_i = 1).
\]
Then the increase of the degree at a birth event is 
\[
\delta = \sum_{i=1}^{\kappa} \gamma_i.
\]
Introduce the notation
\begin{align*}
	\Delta_n = \sum_{i = 1}^{n} \delta_i,
\end{align*}
where the random variables $\delta_i$ are iid copies of $\delta$. Then the reproduction 
process of the monitored vertex's degree is $\eta(t) = \Delta_{\pi(t \wedge \lambda)}$, 
where $(\pi(t))_{t \geq 0}$ and $\lambda$ are the same as before. Note that the biological 
age of an edge still grows by every birth, even though it is not considered in the degree process.

In order to find the condition of supercriticality, and formulate the equation for the Malthusian 
parameter, we can argue as follows. By the definition, $\E(\gamma)=\E(\phi)/2$, $\E(\delta) 
= \E(\veps)/2$, and as a consequence we have
\[
\E \eta(\lambda) = \E(\Delta_{\pi(\lambda)}) = \E(\delta)\E(\lambda) = \frac{1}{2}\,
\E(\veps) \E(\lambda) = \frac{1}{2}\,\E \xi(\infty).
\]
Referring to \eqref{supercritical}, one can immediately see that the degree process is supercritical 
if $\E \eta(\infty) = \E \xi(\infty)/2 > 1$, that is, 
\[
\frac{ \E (\veps)}{c} \int_0^1 u^{\tfrac{1+b}{c}-1} \exp \bigg\{ \frac{1}{c} \int_u^1 
\frac{g_{\veps}(v)}{v} \, dv \bigg\}du > 2.
\]
Similarly, the intensity measure of the degree process of a fixed vertex is just the half of the edge 
process' intensity measure, from which it follows that the Malthusian parameter $\beta > 0$ of the 
former one satisfies the equation
\[
\frac{ \E (\veps)}{c} \int_0^1 u^{\tfrac{\beta+1+b}{c}-1} \exp \bigg\{
\frac{1}{c} \int_u^1 \frac{g_{\veps}(v)}{v} \, dv \bigg\}du = 2.
\]
It is clear that $\beta<\alpha$.

For the probability that the vertex becomes isolated, i.e., the corresponding degree process extincts, 
we need to compute the probability generating function of $\eta(\infty)$, for which we can use
Lemma \ref{genfunc}. In fact, this is only valid in the case when the initial degree og the vertex is $1$.
When the initial degree of the observed vertex is $2$, its degree process is the superposition of two 
independent copies of $(\eta(\cdot), \lambda)$.

\begin{proposition}
The probability generating function of $\eta(\lambda)$ is
\[
1 - \frac{1 -  g_{\kappa, \veps} \big( \tfrac{(1+z)^2}{4z}, \tfrac{2z}{1+z} \big) }{c} 
\int_0^1 u^{ \tfrac{1+b}{c} - 1 } \exp \bigg\{ \int_u^1 g_{\kappa, \veps}
\big( \tfrac{(1+z)^2}{4z}, \tfrac{2sz}{1+z}\big) \,ds \bigg\} du.
\]
Hence, if $\E \eta(\lambda) > 1$, the probability that a fixed vertex eventually gets isolated is equal to 
$pz^2 + (1-p)z$, where $z$ is the smallest positive root of the equation
\begin{align*}
\frac{1 -  g_{\kappa, \veps} \big( \tfrac{(1+z)^2}{4z}, \tfrac{2z}{1+z} \big) }{c} 
\int_0^1 u^{ \tfrac{1+b}{c} - 1 } \exp \bigg\{ \int_u^1 g_{\kappa, \veps}\big( \tfrac{(1+z)^2}{4z}, 
\tfrac{2sz}{1+z} \big) \, ds \bigg\}du = 1-z.
\end{align*}
\end{proposition}

\begin{proof}
By the law of total expectation,
\[
g_{\eta(\lambda)}(z) = \E \big(z^{ \eta(\lambda)} \big) = \E \big[ \E \big( z^{ \eta(\lambda) } 
\big| \pi'(\lambda), \, \xi(\lambda) \big) \big],
\]
therefore it is sufficient to deal with the conditional expectation
\[
\E \big( z^{ \eta(\infty) } \big| \pi'(\lambda) = \ell, \, \xi(\lambda) = k \big)
\]
where $0\le\ell\le k\le 2\ell$. Clearly, if $k$ new edges are added with $\ell$ new vertices, then the 
increment is composed of $k-\ell$ cherries and $2\ell-k$ semi-cherries. Thus the conditional 
distribution of $\eta(\lambda)-(k-\ell)$ is Binomial$(2\ell-k,\,1/2)$. Hence,
\[
\E \big( z^{\eta(\lambda)} \big| \pi'(\lambda) =\ell, \xi(\lambda) = k \big) = z^{k-l} 
\Big( \frac{1+z}{2} \Big)^{2\ell -k}.
\]
Consequently, by Lemma \ref{genfunc} we get
\begin{align*}
& g_{\eta(\lambda)}(z) = \E \bigg[ z^{\xi(\lambda)-\pi'(\lambda)} \Big(\frac{1+z}{2} 
\Big)^{2\pi'(\lambda) -\xi(\lambda)} \bigg]\\
&\quad = f \bigg( \frac{(1+z)^2}{4z}, \frac{2z}{1+z} \bigg)\\ 
&\quad = 1 - \frac{1 -  g_{\kappa, \veps} \big( \tfrac{(1+z)^2}{4z}, \tfrac{2z}{1+z} \big) }{c} 
\int_0^1 u^{ \tfrac{1+b}{c} - 1 } \exp \bigg\{ \int_u^1 g_{\kappa, \veps}
\big( \tfrac{(1+z)^2}{4z}, \tfrac{2sz}{1+z} \big) ds \bigg\} du,
\end{align*}
completing the proof.
\end{proof}

\subsection{Further properties}

The first proposition of the section is about the asymptotic proportion of living edges 
without any descendants.
\begin{proposition}\label{no_descendants}
Let  us denote the number of childless edges after $n$ steps by $O_n$, then
\[
\lim_{n \to \infty} \frac{O_n}{E_n} = \frac{ \E (\veps)}{1+b+\alpha}
\]
almost everywhere on the event of non-extinction. 
\end{proposition}

\begin{proof}
Similarly to what we did in the proofs of previous results, we will introduce the notation $O(t)$ 
for the number of living childless edges at time $t$ in the continuous time model, and since
\[
\lim_{t \to \infty} \frac{O_n}{E_n} = \lim_{n \to \infty} \frac{O(t)}{E(t)}
\]
holds, we can rely on Proposition \ref{living_edges} and Theorem \ref{Nerman} with the 
appropriate random characteristic. It is easy to see that the right choice is 
\[
\phi(t) = \ind_{ \{ 0 \leq t < \tau_1\wedge\lambda \} },
\]
where $\tau_1$ is the first birth time in the Poisson process $(\pi(t))_{t \geq 0}$.
	
To calculate the corresponding limit fraction, we first need to compute the mean $\E \phi(t)
= \Pr(\tau_1\wedge\lambda > t )$.
The distribution of $\tau_1$ is exponential with mean $1$, and up to $\tau_1$, the hazard 
rate of the edge lifetime is constant $b$, that is, $\lambda$ behaves like an exponenetial 
random variable, which is independent of $\tau_1$. Therefore $\tau_1\wedge\lambda$ is 
exponentially distributed with parameter $1+b$,  thus $\E\phi(t)=e^{-(1+b)t}$.

By Theorem \ref{Nerman} we have
\[
\lim_{t \to \infty} \frac{O(t)}{E(t)} =\frac{\int_0^{\infty} e^{- \alpha t} e^{-(1+b)t} \, dt}
{\int_0^{\infty}e^{-\alpha t} S(t)\,dt} = \frac{ \E (\veps)}{1 + b + \alpha}\,.
\qedhere
\]
\end{proof}

So far we could easily utilize the direct connection between the discrete and the continuous 
models. The next example will show a case where the transfer of results is not so 
straightforward.

Consider the continuous cherry tree and define 
\[
T(t)=\big|\{e: \sigma_e\le t\}\big|;
\]
this is the number of edges born up to time $t$, irrespectively that they are still present or already 
deleted. Moreover, let
\[
J(t)=\int_0^t E(s)\,ds,
\]
where $E(s)$ is the number of living edges at time $s$. Clearly,
\[
J(t)=\int_0^t \sum_e \ind_{\{\sigma_e\le s<\sigma_e+\lambda_e\}}\,ds=
\sum_e\int_0^t  \ind_{\{\sigma_e\le s<\sigma_e+\lambda_e\}}\,ds= 
\sum_e (t-\sigma_e)^+\wedge \lambda_e,
\]
thus $J(t)$ is the sum of the lengths of time the edges spent in the graph up to time $t$. In the statistical
analysis of survival data this quantity is called \emph{the total time on test}. The summands can also 
be considered a censored sample from the lifetime distribution $\lambda$, hence the mean lifetime
$\E(\lambda)$ can be estimated by $\hat\lambda_1(t)=J(t)/T(t)$. As a result of censoring, this 
estimation is underbiased. One might reduce the bias by leaving censored observations out of 
consideration. This leads to the estimator $\hat\lambda_2(t)=\widetilde J(t)/\widetilde T(t)$, where
\[
\widetilde J(t)=\sum_{e\,:\, \sigma_e+\lambda_e\le t} \lambda_e,\qquad 
\widetilde T(t)=\big|\{e: \sigma_e+\lambda_e\le t\}\big|.
\]
We should remark that this second estimator is still underbiased, because the exponential 
growth of the continuous cherry tree imples that a non-negligible proportion of the edges born so 
far entered the graph in the recent past, and they are only counted if died at an unusally young age. 

Let us compute their limits as $t\to\infty$. 
\begin{proposition}
Almost everywhere on the event of non-extinction, 
\begin{align*}
\lim_{t\to\infty}\hat\lambda_1(t)=\frac{1}{\E(\veps)}\,, \quad 
\lim_{t\to\infty}\hat\lambda_2(t)=\frac{1-\alpha \E(\veps)\int_0^\infty t\, e^{-\alpha t}S(t)\,dt)}
{\E(\veps)-1}\,.
\end{align*}
\end{proposition}
\begin{proof}
All four quantities can be expressed as $Z^{\phi}(t)$ by the help of suitable random characteristics
$\phi$ as Table \ref{tablazat} shows.
\begin{table}[h]\label{tablazat}
\centering
\setlength{\extrarowheight}{2pt}
\begin{tabular}{|l|l|}
\hline
$J(t)$ & $\phi_1(t)=(t\wedge\lambda)\,\ind_{\{t\ge 0\}}$ \\
$T(t)$ & $\phi_2(t)=\ind_{\{t\ge 0\}}$ \\
$\widetilde J(t)$ & $\phi_3(t)=\lambda\,\ind_{\{\lambda\le t\}}$ \\
$\widetilde T(t)$ & $\phi_4(t)=\ind_{\{\lambda\le t\}}$ \\
\hline
\end{tabular}
\caption{Statistics and the corresponding random characteristics}
\label{tablazat}
\end{table}

First, we have to compute $\E\phi_i(t)$, $i=1,2,3,4$, $t\ge 0$.
\begin{align*}
\E\phi_1(t)&=\E(t\wedge\lambda)=\int_0^t S(s)\,ds, &
\E\phi_2(t)&=1, \\
\E\phi_3(t)&=\E(\lambda\,\ind_{\{\lambda\le t\}})=\int_0^t [S(s)-S(t)]\,ds, &
\E\phi_4(t)&=1-S(t).
\end{align*}
Therefore, by \eqref{kellenifog} and \eqref{intensity} we have
\[
\int_0^\infty e^{-\alpha t}\E\phi_1(t)\,dt=\frac{1}{\alpha}\int_0^\infty e^{-\alpha t}S(t)\,dt
=\frac{1}{\alpha  \E (\veps)}\,.
\]
Obviously,
\begin{align*}
\int_0^\infty e^{-\alpha t}\E\phi_2(t)\,dt &=\frac{1}{\alpha},\\
\int_0^\infty e^{-\alpha t}\E\phi_4(t)\,dt &=\int_0^\infty e^{-\alpha t}(1-S(t))\,dt
=\frac{1}{\alpha} - \frac{1}{\alpha  \E (\veps)}\,,
\end{align*}
and
\begin{align*}
\int_0^\infty e^{-\alpha t}\E\phi_3(t)\,dt &=\int_0^\infty e^{-\alpha t}\int_0^t 
[S(s)-S(t)]\,ds\,dt \\
&= \int_0^\infty e^{-\alpha t}\E\phi_1(t)\,dt - \int_0^\infty t\, e^{-\alpha t} S(t)\,dt\\
&=\frac{1}{\alpha  \E (\veps)}-\int_0^\infty t\, e^{-\alpha t} S(t)\,dt\\
&=\frac{1}{\alpha  \E (\veps)}\Bigg(1-\alpha \int_0^\infty t\, e^{-\alpha t}\mu(dt)\Bigg),
\end{align*}
using that $\E(\veps)S(t)\,dt=\mu(dt)$. Application of Theorem \ref{Nerman} will complete the proof.
\end{proof}
\begin{remark}
Unfortunately, the last integral cannot be transformed into a closed form, but
 we find the following connection between $T(t)$ and $\widetilde J(t)$:
\[
\lim_{t\to\infty} e^{-\alpha t}T(t)=Y_\infty+\E(\veps) \lim_{t\to\infty}e^{-\alpha t}\widetilde J(t)
\]
almost everywhere on the event of non-extinction.
\end{remark}

Next, let us turn to the discrete time cherry tree. Though the number $T_n$ of edges born up to time $t$
is of order $n$, the total time on test statistic $J_n=E_0+E_1+\dots+E_n$ exhibits a completely
different behavour. By Corollary \ref{edges} we have
\begin{equation}\label{J_n}
J_n\sim\frac{\alpha}{\E(\veps)+1-\alpha}\sum_{i=0}^n i\sim\frac{\alpha n^2}{2(\E(\veps)+1-\alpha)}\,,
\end{equation}
thus $J_n/T_n$ tends to infinity on the event of non-extinction. This is not at all surprising, because
the correspondence of the discrete and continuous time models is based on a time transform, 
by which the discrete time model is a slowed down version of the continuous one. The later an edge
is born, the longer its life will last. If we want to infer from a continuous counterpart, time has to be 
measured by the number of events; that is,  instead of $J(t)$ we should use $\int_0^t E(s)\,dH(s)$.
Using Corollary \ref{edges} and integrating by parts we get
\begin{align*}
\int_0^t E(s)\,dH(s)&\sim \int_0^t e^{-\alpha(t-s)}E(t)\,dH(s)\\
&=E(t)\bigg(\Big[e^{-\alpha(t-s)}H(s)\Big]_{s=0}^t-\int_0^t \alpha e^{-\alpha(t-s)}H(s)\,ds\bigg)\\
&\sim E(t)\bigg(H(t)-\int_0^t \alpha e^{-2\alpha(t-s)}H(t)\,ds\bigg)\\
&\sim\tfrac{1}{2} E(t)H(t),
\end{align*}
which already corresponds to \eqref{J_n}.
 
In the discrete model it seems more adequate to measure an edge's lifetime by the number of birth events,
which is not affected by time transformations. The mean number of litters during the life of an edge can be 
esimated by the statistic $B_n/T_n$, where $B_n$ is the number of reproduction events in the
first $n$ steps. The corresponding quantity in the continuous model can be counted by 
the random characteristic  $\phi(t)=\pi(t\wedge\lambda)$. Hence,
\[
\lim_{n\to\infty}\frac{B_n}{T_n}=\lim_{t\to\infty}\frac{\int_0^t e^{-\alpha t}\E(t\wedge\lambda)\,dt}
{\int_0^\infty e^{-\alpha t}\,dt}=\frac{1}{\E(\veps)};
\]
this coincides with the limit of $\hat\lambda_1(t)$.

\section*{Acknowledgement}
This project was supported by the European Union, co-financed by the European Social Fund
(EFOP-3.6.3-VEKOP-16-2017-00002). T.\ F.\ M\'ori was also supported by the Hungarian
National Research, Development and Innovation Office NKFIH (K 125569).

\noindent\textbf{Tam\'as F. M\'ori}\\
Department of Probability Theory and Statistics,\\
ELTE E\"otv\"os Lor\'and University,\\
P\'azm\'any P. s. 1/C, H-1117 Budapest, Hungary\\
\textit{e-mail:} \texttt{mori@math.elte.hu}

\smallskip
\noindent\textbf{S\'andor Rokob}\\
Department of Probability Theory and Statistics,\\
ELTE E\"otv\"os Lor\'and University,\\
P\'azm\'any P. s. 1/C, H-1117 Budapest, Hungary\\
\textit{e-mail:} \texttt{rodnasbokor@gmail.com}

\end{document}